\documentclass[11pt]{amsart}
\usepackage[margin=1in]{geometry}
\usepackage{
 amsmath, 
 amsxtra, 
 amsthm, 
 amssymb, 
 etex, 
 mathrsfs, 
 mathtools, 
 tikz-cd, 
 bbm,
 xr,
 comment,
 enumitem}

\usepackage{lipsum} 

\usepackage{tikz,xcolor,hyperref}

\definecolor{lime}{HTML}{A6CE39}
\DeclareRobustCommand{\orcidicon}{%
	\begin{tikzpicture}
	\draw[lime, fill=lime] (0,0) 
	circle [radius=0.16] 
	node[white] {{\fontfamily{qag}\selectfont \tiny ID}};
	\draw[white, fill=white] (-0.0625,0.095) 
	circle [radius=0.007];
	\end{tikzpicture}
	\hspace{-2mm}
}

\foreach \x in {A, ..., Z}{%
	\expandafter\xdef\csname orcid\x\endcsname{\noexpand\href{https://orcid.org/\csname orcidauthor\x\endcsname}{\noexpand\orcidicon}}
}

\usepackage[normalem]{ulem}
\usepackage[toc]{appendix} 
\usepackage[all]{xy}
\usepackage{hyperref}
\usepackage{url}
\usepackage{tipa}
\usepackage{dsfont}
\usepackage{ textcomp }
\usepackage{stmaryrd} 
\usepackage{amssymb}
\usepackage{mathabx} 
\usepackage{amsmath} 
\usepackage{amscd}
\usepackage{amsbsy}
\usepackage{comment, enumerate}
\usepackage{color}
\usepackage{mathtools,caption}
\usepackage{tikz-cd}
\usepackage{longtable}
\usepackage[utf8]{inputenc}
\usepackage[OT2,T1]{fontenc}
\usepackage{changepage}
\usepackage{commath}

\DeclareMathOperator{\Tw}{Tw}

\DeclareMathOperator{\Hom}{Hom}
\DeclareMathOperator{\Ext}{Ext}

\DeclareMathOperator{\Gal}{Gal}

\DeclareMathOperator{\Sel}{Sel}

\DeclareMathOperator{\cyc}{cyc}

\DeclareMathOperator{\Iw}{Iw}

\newtheorem{theorem}{Theorem}[section]
\newtheorem*{theorem*}{Theorem}
\newtheorem{lemma}[theorem]{Lemma}

\newtheorem{corollary}[theorem]{Corollary}

\numberwithin{equation}{section}

\theoremstyle{remark}
\newtheorem{remark}[theorem]{Remark}

\newcommand\EatDot[1]{}

\renewcommand{\Col}{\mathrm{Col}}

\newcommand{\vp}{\varphi}

\newcommand{\ff}{\mathfrak{F}}
\newcommand{\fZ}{\mathfrak{H}^1_{f,i}}
\newcommand{\tZ}{\mathfrak{H}^2_{f,i}}
\newcommand{\fZgj}{\mathfrak{H}^1_{g,j}}

\newcommand{\Image}{\mathrm{Image}}
\newcommand{\loc}{\mathrm{loc}}

\newcommand{\Cp}{\mathbb{C}_p}
\newcommand{\cZ}{\mathbb{H}^1_{f,i}}
\newcommand{\cZgj}{\mathbb{H}^1_{g,j}}

\newcommand{\QQ}{\mathbb{Q}}
\newcommand{\ZZ}{\mathbb{Z}}
\newcommand{\Qp}{\mathbb{Q}_p}
\newcommand{\Zp}{\mathbb{Z}_p}
\newcommand{\NN}{\mathbb{N}}
\newcommand{\lb}{\llbracket}
\newcommand{\rb}{\rrbracket}
\newcommand{\HIw}{H^1_{\mathrm{Iw}}}

\definecolor{Green}{rgb}{0.0, 0.5, 0.0}
\newcommand{\bZ}{\mathbf{z}}
\newcommand{\cO}{\mathcal{O}}
\newcommand{\cC}{\mathcal{C}}

\newcommand{\Qcyc}{\QQ_{\cyc}}
\newcommand{\Qcycp}{\QQ_{\cyc,p}}

\newcommand{\Char}{\mathrm{Char}}
\newcommand{\tor}{{\Lambda-\mathrm{tor}}}

\newcommand{\of}{{\overline{f}}}
\title[Fine and signed Selmer groups]{On fine Selmer groups and signed Selmer groups of elliptic modular forms}

\makeatletter
\let\@wraptoccontribs\wraptoccontribs
\makeatother

\author[A. Lei]{Antonio Lei\orcidA{}}
\address[Lei]{D\'epartement de Math\'ematiques et de Statistique\\
Universit\'e Laval, Pavillion Alexandre-Vachon\\
1045 Avenue de la M\'edecine\\
Qu\'ebec, QC\\
Canada G1V 0A6}
\email{antonio.lei@mat.ulaval.ca}
\author[M.F.~Lim]{Meng Fai Lim\orcidB{}}
\address[Lim]{School of Mathematics and Statistics \& Hubei Key Laboratory of Mathematical Sciences\\ Central China Normal University\\ Wuhan\\ 430079\\ P.R.China.}
\email{limmf@mail.ccnu.edu.cn}
\thanks{The authors' research is partially supported by:  the NSERC Discovery Grants Program RGPIN-2020-04259 and RGPAS-2020-00096 (Lei), the National Natural Science Foundation of China under Grant No.\ 11771164 and the Fundamental Research Funds for the Central Universities of CCNU under grant CCNU20TD002 (Lim). 
}


\begin{document}
\begin{abstract}
Let $f$ be an elliptic modular form and $p$ an odd prime that is coprime to the level of $f$. We study the link between divisors of the characteristic ideal of the $p$-primary fine Selmer group of $f$ over the cyclotomic $\mathbb{Z}_p$ extension of $\mathbb{Q}$ and the greatest common divisor of signed Selmer groups attached to $f$ defined using the theory of Wach modules. One of the key ingredients of our proof is a  generalization of a result of Wingberg on the structure of fine Selmer groups of abelian varieties with supersingular reduction at $p$ to the context of modular forms.
\end{abstract}

\maketitle

\section{Introduction}

Let $p$ be a fixed odd prime number and $f$ a fixed normalized eigen-cuspform of level $N$ and weight $k\ge 2$, with $p\nmid N$. Let $K$ denote the completion of the Hecke field of $f$ at a prime above $p$ and write $\cO$ for its ring of integers. Let  $V_f$ denote the $K$-adic $G_\QQ$-representation of $f$ defined by Deligne \cite{deligne69}. We fix a Galois-stable $\cO$-lattice $T_f$ inside $V_f$ and write $A_f=V_f/T_f$.
We write $\of$ for the conjugate modular form of $f$ and we write $V_\of$, $T_\of$ and $A_\of$ for the corresponding $G_\QQ$-modules attached to $\of$ with  $T_\of$ chosen to be $\Hom_\cO(T_f,\cO)(1-k)$, where $M(j)$ denotes the $j$-th Tate twist of a $G_\QQ$-module $M$.

Let $\Qcyc$ be the cyclotomic $\Zp$-extension of $\QQ$ and let $\Gamma$ denote the Galois group $\Gal\left(\Qcyc/\QQ\right)$.
The Iwasawa algebra $\Lambda=\cO\lb\Gamma\rb $ is defined to be $\varprojlim \cO[\Gamma/\Gamma^{p^n}]$, where the connecting maps are projections.
After fixing a topological generator $\gamma$ of $\Gamma$, there is an isomorphism of rings $\Lambda\cong\cO\lb X\rb $, sending $\gamma$ to $X+1$.
Given a $\Lambda$-module $M$, denote its Pontryagin dual by $M^\vee := \Hom_{\cO}\left( M, K/\cO\right)$. {We write $M^{\iota}$ for the $\Lambda$-module which is $M$ as an $\cO$-module equipped with a $\Gamma$-action given by $\gamma\cdot_{\iota} m = \gamma^{-1} m$. Finally, if $F\in\Lambda= \cO\lb X\rb$, we write $F^\iota$ for the power series  $F(\frac{1}{1+X}-1)$.}

 For $g=f$ or $\of$ and $j\in\ZZ$, let $\Sel_0(A_g(j)/\Qcyc)$ denote the fine Selmer group of $A_g(j)$ over $\Qcyc$ (whose definition will be reviewed in \S\ref{S:fine}). The fine Selmer groups of abelian varieties were first systematically studied by Coates and Sujatha in \cite{CoatesSujatha_fineSelmer}, and a little later by Wuthrich \cite{Wut-JAG}. Various conjectures on the structure of these groups have been formulated and they are still wide open to this date. See also \cite{KS,leisuj, lim2020control, LM16, matarfine} for some recent results on the topic. In the context of modular forms, the fine Selmer groups have been studied in \cite{aribam2014mu,Jha12,jhasuj,HKLR}. 
 
 It has been proved by Kato \cite{Kato} that $\Sel_0(A_g(j)/\Qcyc)^\vee$ is a finitely generated torsion $\Lambda$-module. 
Let $\ff_{g,j}$ denote a choice of characteristic element (i.e. a generator of the characteristic ideal $\Char_\Lambda\Sel_0(A_g(j)/\Qcyc)^\vee$). The main goal of the present article is to study the divisors of $\ff_{f,i}$ and $\ff_{\of,k-i}$ for a fixed integer $i$. These specific twists are considered due to the perfect pairing
\[
T_f(i)\times A_\of(k-i)\rightarrow \mu_{p^\infty}
\]
of $G_\QQ$-modules.

We are able to relate  the divisors of $\ff_{f,i}$ and $\ff_{\of,k-i}$ to the greatest common divisors of the signed Selmer groups attached to  $\of(k-i)$ defined in \cite{LLZ0,LLZ0.5} by the theory of Wach modules (see \S\ref{S:signed-Sel} where the definition of these groups are reviewed). Let us write $\ff_{\of,k-i}^\sharp$ and $\ff_{\of,k-i}^\flat$ for a choice of characteristic elements of the Pontryagin duals of these signed Selmer groups. Under a mild hypothesis on the local representations at bad primes and a hypothesis on the validity of Kato's Iwasawa main conjecture for $\of(k-i)$ (labelled \textbf{(H0)} and \textbf{(H-IMC}) respectively in the main text), we show that if $F$ is an irreducible element of $\Lambda$ that is outside a certain explicit set of linear factors,   $F^\iota\nmid \ff_{f,i}$ and  $F\nmid \ff_{\of,k-i}$   if and only if  $F\nmid\gcd\left(\ff_{\of,k-i}^\sharp,\ff_{\of,k-i}^\flat\right)$. (See Theorem~\ref{thm:LS} for the precise statement of this result.)

In the case where $T_f$ is the Tate module of an elliptic curve $E/\QQ$ with good supersingular reduction at $p$, a similar result has been proved in \cite{leisuj}. One of the key ingredients of the proof given in \cite{leisuj} is a link between the fine Selmer group of $E$ over $\Qcyc$ and the maximal $\Lambda$-torsion submodule of  the Pontryagin dual of the $p$-primary  Selmer group of $E$ over $\Qcyc$, which was  proved by Wingberg  \cite{wingberg} (see also \cite{matar}, where an alternative proof is given). In the present paper, we prove Theorem~\ref{thm:LS} by first establishing analogues of Wingberg's result in the context of modular forms (see Theorems~\ref{thm:Wing} and \ref{thm:Wing2}). In \cite{wingberg}, Wingberg worked with Selmer groups defined using flat cohomology,  whereas Matar worked with Selmer groups defined using Galois cohomology in \cite{matar} instead. In the present paper,  the latter definition is used.  Our proof is very different from the ones employed in both \cite{wingberg} and \cite{matar}. We make use of Nekov\'a\v{r}'s spectral sequence, which seems to give a somewhat simpler and more general proof than the previous proofs available in the literature. We hope that these results may be of independent interest.

\section{Definitions of Selmer groups and related objects}
\subsection{Fine Selmer groups}\label{S:fine}
Let $L_n$ be the unique subextension of $\Qcyc$ such that $[L_n:\QQ]=p^n$. Given an algebraic extension $L$ of $\QQ$, we write $S(L)$ for the primes of $L$ lying above $pN$ as well as the archimedean primes. We write $G_S(L)$ for the Galois group of the maximal algebraic extension of $L$ that is unramified outside $S(L)$.

For $g\in\{f,\of\}$ and $j\in\ZZ$, we define the fine Selmer group of $A_g(j)$ over $L$ to be
\[
\Sel_0(A_g(j)/L)=\ker\left(H^1(G_S(L),A_g(j))\rightarrow\prod_{v\in S(L)} H^1(L_v,A_g(j))\right).
\]
Recall that $\Sel_0(A_g(j)/\Qcyc)^\vee$ is torsion over $\Lambda$ and we write $\ff_{g,j}$ for a choice of a generator of the characteristic ideal $\Char_\Lambda\Sel_0(A_g(j)/\Qcyc)^\vee$.

The classical Selmer group of $A_g(j)$ over $L$ is defined to be
\[
\Sel(A_g(j)/L)=\ker\left(H^1(G_S(L),A_g(j))\rightarrow\prod_{v\in S(L)} \frac{H^1(L_v,A_g(j))}{H^1_\mathrm{f}(L_v,A_g(j))}\right),
\]
where $H^1_\mathrm{f}(L_v,A_g(j))$ is defined as in \cite[\S3]{blochkato}.
\subsection{Wach modules and signed Coleman maps}
We recall the definition of signed Coleman maps from \cite{LLZ0,LLZ0.5}, which generalize those studied in \cite{kobayashi03,sprung09} in the context of elliptic curves with supersingular reduction at $p$. Here, we do not require $p$ to be a non-ordinary prime for $f$. The only requirement is that $p$ is coprime to $N$ so that the representation $V_f$ is crystalline at $p$.

We shall write for $g\in\{f,\of\}$, $j\in\ZZ$ and $m\in\{1,2\}$,
\[
\mathfrak{H}^m_{g,j}=\varprojlim_n H^m(L_n,T_g(j)), \quad  \mathbb{H}^m_{g,j}=\varprojlim H^m(L_{n,p},T_g(j)),
\]
where the connecting maps are corestrictions and we have abused notation by writing $p$ for the unique prime of $L_n$ above $p$.
We write $\loc_p:\fZgj\rightarrow \cZgj$ for the localization map.

Let $T=T_g(k-1)$, where $g\in\{f,\of\}$. We write $\NN(T)$ for the Wach module of $T$. See \cite[\S II.1]{berger04} for the precise definition of $\NN(T)$. Roughly speaking, it is a filtered $\vp$-module over the ring $\cO\lb\pi\rb$, where $\pi$ is an element in the ring of of Witt vectors of $\displaystyle\varprojlim_{x\mapsto x^p}\Cp$,  given by $[(1, \zeta_p,\zeta_{p^2},...)]-1$ and $\zeta_{p^n}$ is a primitive $p^n$-th root of unity in $\Cp$ such that $\zeta_{p^{n+1}}^p=\zeta_{p^n}$. One may regard $\pi$ as a formal variable  equipped with an action of $\vp$ and $\Gamma_0=\Gal(\Qp(\mu_{p^\infty})/\Qp)$ given by $\vp(\pi)=(1+\pi)^p-1$ and $\sigma(\pi)=(1+\pi)^{\chi_{\cyc}(\sigma)}-1$ for $\sigma\in\Gamma_0$,  where $\chi_{\cyc}$ is the $p$-adic cyclotomic character.  Let  $\HIw(\Qp(\mu_{p^\infty}),T)=\varprojlim_n H^1(\Qp(\mu_{p^n}),T)$, where the connecting maps are corestrictions.  Berger proved in \cite[Appendix A]{berger03} that there is an isomorphism of $\cO\lb\Gamma_0\rb$-modules
\[
\HIw(\Qp(\mu_{p^\infty}),T)\cong \NN(T)^{\psi=1},
\]
where $\psi$ is a canonical left inverse of $\vp$ and the superscript $\psi=1$ denotes the kernel of $\psi-1$.

We recall from  \cite{LLZ0,LLZ0.5} that for  $\bullet\in\{\sharp,\flat\}$, on choosing an appropriate basis of $\NN(T_g(k-1))$, one can define signed Coleman maps
\[
\Col_g^\bullet: \HIw(\Qp(\mu_{p^\infty}),T_g(k-1))\rightarrow \cO\lb\Gamma_0\rb,
\]
decomposing Perrin-Riou's big logarithm map defined in \cite{perrinriou94}. 

We   define the twisted Coleman maps
\[
\Col_{g,j}^\bullet:\cZgj\rightarrow \Lambda
\]
as follows. We have the $\cO$-isomorphism
\[
\HIw (\Qp(\mu_{p^\infty}),T_g(j))\stackrel{ e^{k-1-j}}{\longrightarrow} \HIw(\Qp(\mu_{p^\infty}),T_g(k-1)),
\]
where $e$ is a basis of the $G_{\Qp}$-representation $\Zp(1)$, as given in \cite[\S A.4]{perrinriou95}. Let $\Tw$ be the $\cO$-linear automorphism on $\cO\lb\Gamma_0\rb$ sending $\sigma\in \Gamma_0$ to $\chi_{\cyc}(\sigma)\sigma$. We can define a $\Lambda$-morphism
\[
\Tw^{i-k+1}\circ\ \Col_g^\bullet\circ e^{k-1-j}: \HIw(\Qp(\mu_{p^\infty}),T_g(j))\rightarrow \cO\lb\Gamma_0\rb
\]
and obtain $\Col_{g,j}^\bullet$ on taking the trivial isotypic component of $\Delta:=\Gal(\Qp(\mu_p)/\Qp)$.
\subsection{Signed Selmer groups and main conjectures}\label{S:signed-Sel}
Let $H_\bullet^1(\Qcycp,A_\of(k-i))$ denote the orthogonal complement of $\ker\Col_{f,i}^\bullet$ under the local Tate pairing
\[
\HIw(\Qp(\mu_{p^\infty}),T_f(i))\times H^1(\Qcycp,A_\of(k-i))\rightarrow \Qp/\Zp.
\]
We  define the signed Selmer groups $\Sel^\bullet(A_\of(k-i)/\Qcyc)$ as
\[
\ker\left(H^1(G_S(\Qcyc),A_\of(k-i))\rightarrow\frac{ H^1(\Qcycp,A_\of(k-i))}{H_\bullet^1(\Qcycp,A_\of(k-i))}\times \prod_{v\nmid p}\frac{ H^1(\QQ_{\cyc,v},A_\of(k-i))}{H_\mathrm{f}^1(\QQ_{\cyc,v},A_\of(k-i))}\right).
\]
Note that when  $A_\of(k-i)$ is given by $E[p^\infty]$ for some elliptic curve $E/\QQ$ with good supersingular reduction at $p$ or when $a_p(\of)=0$, one may choose an appropriate basis of the Wach module so that these Selmer groups coincide with the ones studied in \cite{kobayashi03,sprung09,lei11compositio}  (see \cite[\S\S5.2-5.4]{LLZ0}).

Let $\bZ_{f,i}$ be Kato's zeta element inside $\HIw(\Qcyc,T_f(i))\otimes_{\Zp}\Qp$ as defined in \cite{Kato}. It is conjectured that $\bZ_{f,i}\in\HIw(\Qcyc,T_f(i))$ (see Conjecture~12.10 of op. cit.). Assuming this (which we shall do for the rest of the article), recall from \cite[Eqn. (61)]{LLZ0}  the Poitou--Tate exact sequence
\begin{equation}\label{eq:PT-signed}
    0\rightarrow\frac{\fZ}{\Lambda \bZ_{f,i}}\rightarrow\frac{\Image\left(\Col_{f,i}^\bullet\right)}{(L_p^\bullet)}\rightarrow \Sel^\bullet(A_\of(k-i)/\Qcyc)^\vee\rightarrow\Sel_0(A_\of(k-i)/\Qcyc)^\vee\rightarrow 0,
\end{equation}
where $L_p^\bullet=\Col_{f,i}^{\bullet}\circ \loc_p(\bZ_{f,i})$ is the signed $p$-adic $L$-function associated to $\of(k-i)$.
If $L_p^\bullet\ne0$, then $\Sel^\bullet(A_\of(k-i)/\Qcyc)^\vee$ is a finitely generated torsion $\Lambda$-module. We write $\ff_{\of,k-i}^\bullet$  for a characteristic element of this module.

Kato \cite{Kato} proved that the first and fourth terms of \eqref{eq:PT-signed} are $\Lambda$-torsion. He further formulated the following Iwasawa main conjecture:
\begin{equation}\label{eq:Kato}
  \Char_\Lambda\fZ/\Lambda \bZ_{f,i}=\Char_\Lambda\Sel_0(A_\of(k-i)/\Qcyc)^\vee=(\ff_{\of,k-i}). 
\end{equation}
 It follows from \eqref{eq:PT-signed} that it is equivalent to
\begin{equation}\label{eq:IMC}
\Char_\Lambda\Sel^\bullet(A_\of(k-i)/\Qcyc)^\vee=\Char_\Lambda\Image\left(\Col_{f,i}^\bullet\right)/(L_p^\bullet),
\end{equation}
provided that $L_p^\bullet\ne 0$. 

For the rest of the article, we shall assume that the following hypothesis holds:

\vspace{0.25cm}
\noindent \textbf{(H-IMC)} Kato's Iwasawa main conjecture \eqref{eq:Kato} holds and that both $L_p^\sharp$ and $L_p^\flat$ are non-zero.
\vspace{0.25cm}

\begin{remark}
We say a few words on the hypothesis \textbf{(H-IMC)}. In his seminal work \cite{Kato}, Kato  established that the inclusion ``$\subseteq$'' in \eqref{eq:Kato} holds after tensoring by $\Qp$ under the assumption that the image of the representation $T_f|_{G_{\QQ(\mu_{p^\infty})}}$ contains a copy of $\mathrm{SL}_2(\Zp)$, provided that certain local terms vanish (see in particular Theorem~12.5(4) of op. cit.). Note that in \cite[Conjecture~12.10]{Kato}, Kato's main conjecture is formulated in terms of étale cohomology groups. Kobayashi and Kurihara showed that it can be recast in terms of Galois cohomology groups and fine Selmer groups  (see for example \cite[Proposition~7.1]{kobayashi03}).  The reverse inclusion of \eqref{eq:Kato} has been established in the monumental work of Skinner-Urban \cite{skinnerurbanmainconj} for a $p$-ordinary modular form (under certain hypotheses, see Theorem~1 of op. cit. for details). In the non-ordinary form, there have been several recent breakthroughs in this reverse direction (for instance, see \cite{CCSsprung,castellaLiuWan,FW}).
Finally, the work \cite{LLZ0} has supplied many sufficient conditions  for the non-vanishing of $L_p^\sharp$ and $L_p^\flat$ (see in particular Corollary~3.29 and Proposition~3.39 in op. cit.). In view of these developments, it seems reasonable to assume that the hypothesis \textbf{(H-IMC)}, on which our main results are reliant,  holds.
\end{remark}

\subsection{Images of signed Coleman maps}

We review an explicit description of the images of these Coleman maps. Recall that $\gamma$ is a fixed topological generator of $\Gamma$ giving the identification $\Lambda=\cO\lb X\rb$ via $X=\gamma-1$. Let $u=\chi_{\cyc}(\gamma)$. It was proved in \cite[\S5A]{LLZ0.5}\footnote{In loc. cit., it is assumed that $f$ is non-ordinary at $p$, meaning that $a_p(f)$ is a non-unit in $\cO$. But the same calculations still apply to the ordinary case. The only difference is that the linear relations describing the images of the Coleman maps will be slightly different. See Remark 1.10 in \cite{LLZ0.5}.} that if $\eta$ is a character on $\Delta$, then there exist constants $c_{\eta,j}\in K$ such that
\[
e_\eta\Image(\Col_{f}^\sharp\oplus \Col_f^\flat)\otimes_\cO K=\{(F,G)\in \Lambda^2:F(u^j-1)=c_{\eta,j}G(u^j-1),0\le j\le k-2\}.
\]
In particular, this tells us that there is an isomorphism of $\Lambda$-modules
\[\frac{\Lambda^2\otimes_\cO K}{
e_\eta\Image(\Col_{f}^\sharp\oplus \Col_f^\flat)\otimes_\cO K}\cong \frac{\Lambda\otimes_\cO  K}{\prod_{j=0}^{k-2}(X-u^j+1)},
\]
where $e_\eta$ denotes the idempotent attached to $\eta$.
Integrally, we have:
\begin{lemma}\label{lem:image-pair}
There is a pseudo-isomorphism of $\Lambda$-modules
\[
\frac{\Lambda^2}{e_\eta\Image(\Col_{f}^\sharp\oplus \Col_f^\flat)}\sim \frac{\Lambda}{\prod_{j=0}^{k-2}(X-u^j+1)}.
\]
\end{lemma}
\begin{proof}
It is enough to show that the $\mu$-invariant of the quotient on the left-hand side vanishes. For simplicity, let us write 
\[
C=e_\eta\Image(\Col_{f}^\sharp\oplus \Col_f^\flat),\quad\text{and}\quad C_\bullet=e_\eta\Image\Col_f^\bullet
\]
for $\bullet\in\{\sharp,\flat\}$.

Consider the tautological short exact sequence
\[
0\rightarrow C\rightarrow \Lambda^2\rightarrow \Lambda^2/C\rightarrow 0.
\]
This gives the exact sequence
\[
0\rightarrow\left(\Lambda^2/C\right)[\varpi]\rightarrow C/\varpi\stackrel{\Phi}{\longrightarrow} \Lambda^2/\varpi,
\]
where $\varpi$ is a fixed uniformizer of $\cO$. Similarly, we have the exact sequences
\[
0\rightarrow\left(\Lambda/C_\bullet\right)[\varpi]\rightarrow C_\bullet/\varpi\stackrel{\Phi_\bullet}{\longrightarrow} \Lambda/\varpi,
\]
for $\bullet\in\{\sharp,\flat\}$.

Recall from \cite[Theorem~5.10]{LLZ0.5} that the $\mu$-invariants of $\Lambda/C_\bullet$ are zero. Therefore, $\ker\Phi_\bullet=(\Lambda/C_\bullet)[\varpi]$ are finite for both choices of $\bullet$. Note that $C\subset C_\sharp\oplus C_\flat$ and $\Phi=(\Phi_\sharp\oplus \Phi_\flat)|_{C/\varpi}$ by definitions. It implies that
\[
\ker\Phi\subset \ker\Phi_\sharp\oplus\ker\Phi_\flat.
\]
Hence, $\ker\Phi=\left(\Lambda^2/C\right)[\varpi]$ is finite. In particular, the $\mu$-invariant of $\Lambda^2/C$ is zero, which finishes the proof of the lemma.
\end{proof}
\begin{remark}
In the case where $T_f$ is the $p$-adic Tate module of an elliptic curve with good supersingular reduction at $p$, we can in fact describe the set $C$ explicitly. See \cite[Proposition~2.2]{leilim}.
\end{remark}

It follows from Lemma~\ref{lem:image-pair} that there is an exact sequence
\begin{equation}
    0\rightarrow \cZ\rightarrow \Lambda^2\rightarrow \cC_{k,i}\rightarrow 0,
\label{eq:Col-pair}
\end{equation}
where the first map is given by $\Col_{f,i}^\sharp\oplus\Col_{f,i}^\flat$  and $\cC_{k,i}$ is a $\Lambda$-module which is pseudo-isomorphic to $\Lambda/\eta_{k,i}$, with $\eta_{k,i}$ being the image of $\prod_{j=0}^{k-2}(X-u^j+1)$ under the twisting map $\Tw^{i-k+1}$.

\section{A generalization of a result of Wingberg}
Let $\mathcal{A}$ be an abelian variety over $\QQ$ with supersingular reduction at $p$. Wingberg proved that there is a pseudo-isomorphism of $\Lambda$-modules
\[
\left(\Sel(\mathcal{A}[p^\infty]/\Qcyc)^{\vee}\right)_{\tor}\sim\left(\Sel_0(\mathcal{A}^t[p^\infty]/\Qcyc)^{\vee}\right)^\iota,
\]
where $\mathcal{A}^t$ is the dual abelian variety of $\mathcal{A}$ (see  \cite[Corollary~2.5]{wingberg}).  Here, $M_\tor$ denotes the maximal torsion submodule of  a $\Lambda$-module $M$, $\Sel(\mathcal{A}[p^\infty]/\Qcyc)$ is the $p$-primary Selmer group of $\mathcal{A}$ over $\Qcyc$ and $\Sel_0(\mathcal{A}^t[p^\infty]/\Qcyc)$ is the $p$-primary fine Selmer group of $\mathcal{A}^t$ over $\Qcyc$ defined in a similar manner as the fine Selmer groups for $A_g(j)$ in \S\ref{S:fine} above.
We prove the following analogue of Wingberg's result in the context of modular forms.
\begin{theorem}\label{thm:Wing}
We have a pseudo-isomorphism of $\Lambda$-modules
\[
\left(H^1(\Qcyc,A_g(j))^\vee\right)_\tor\sim \left(\mathfrak{H}^2_{g,j}\right)^\iota,
\] where $g\in\{f,\of\}$ and $j\in\ZZ$.
\end{theorem}

\begin{proof}
From the low degree terms of Nekov\'a\v{r}'s spectral sequence (see \cite[Lemma 9.1.5]{nekovar06})
\[ \Ext_{\Lambda}^{r}( H^{3-s}_{\Iw}(\Qcyc, T_g(j)), \Lambda)\Rightarrow H^{3-r-s}(\Qcyc, A_g(j))^\vee, \]
we obtain the following exact sequence of $\Lambda$-modules:
\[ 0 \longrightarrow  \Ext_{\Lambda}^{1}( H^2_{\Iw}(\Qcyc, T_g(j)), \Lambda)\longrightarrow  H^{1}(\Qcyc, A_g(j))^\vee \longrightarrow  \Hom_{\Lambda}( H^1_{\Iw}(\Qcyc, T_g(j)), \Lambda).\]
Since $\Hom_{\Lambda}( H^1_{\Iw}(\Qcyc, T_g(j)), \Lambda)$ is a reflexive $\Lambda$-module by \cite[Corollary 5.1.3]{Neukirch}, it is $\Lambda$-torsion-free. Hence, it follows from the above exact sequence that there is an isomorphism 
\[
\left(H^1(\Qcyc,A_g(j))^\vee\right)_\tor\cong \Ext_{\Lambda}^{1}( H^2_{\Iw}(\Qcyc, T_g(j)), \Lambda)
\]
of $\Lambda$-modules. By \cite[Proposition 5.5.13]{Neukirch}, the latter is pseudo-isomorphic to $\left(\mathfrak{H}^2_{g,j}\right)^\iota$, which concludes the proof of the theorem.
\end{proof}

When $f$ is non-ordinary at $p$, we shall establish a direct analog of Wingberg's theorem  on the level of Selmer groups under the following hypothesis on the local representation $A_\of$ (see Theorem~ \ref{thm:Wing2} below). 
 
\vspace{0.25cm}
\noindent \textbf{(H0)} For all $v\in S(\Qcyc)$, the group $H^0(\QQ_{\cyc,v},A_\of(k-i))$ is finite.
\vspace{0.25cm}
\begin{remark}
Note that when $v|p$, the group  $H^0(\QQ_{\cyc,v},A_\of(k-i))$ is always finite by \cite[Lemma~3.3]{HKLR}. See also \S5 of op. cit. where sufficient conditions and explicit examples of the finiteness of $H^0(\QQ_{\cyc,v},A_\of(k-i))$ for $v|N$  are studied.
\end{remark}
\begin{lemma}\label{lem:H2-Sel}
Under \textbf{\upshape{(H0)}}, we have a pseudo-isomorphism of $\Lambda$-modules
\[
\Sel_0(A_f(i)/\Qcyc)^\vee\sim \mathfrak{H}^2_{\of,k-i}.
\]
\end{lemma}
\begin{proof}
Let $n\ge0$ be an integer.
By the Poitou--Tate exact sequence (see for example \cite[\S A.3.1]{perrinriou95}), we have the following exact sequence:
\begin{align*}
 \bigoplus_{v \in S(L_n)}H^0\left(L_{n,v}, A_\of(k-i)\right)\longrightarrow H^2\left(G_S(L_n), T_f(i)\right)^{\vee}&\stackrel{\delta}{\longrightarrow} \\
 H^1\left(G_S(L_n), A_\of(k-i)\right)& \stackrel{\tau}{\longrightarrow} \bigoplus_{v \in S(L_n)}H^1\left(L_{n,v},A_\of(k-i)\right).
\end{align*} 
By definition,
\[
\Image(\delta)=\ker(\tau)=\Sel_0\left(A_\of(k-i)/L_n\right).
\]
This gives the exact sequence
\[
\bigoplus_{v \in S(L_n)}H^0\left(L_{n,v}, A_\of(k-i)\right) \longrightarrow H^2\left(G_S(L_n), T_f(i)\right) ^{\vee} \longrightarrow \Sel_0\left(A_\of(k-i)/L_n\right) \longrightarrow 0.\]
By the local Tate duality, we have
\[H^0(L_{n,v}, A_\of(k-i))^{\vee} \cong H^2(L_{n,v}, T_f(i))\]
for all $v\in S$.
Thus, after taking Pontryagin duals and inverse limits, we obtain
\begin{equation}\label{eq:fine-H2}
0\longrightarrow\Sel_0\left(A_\of(k-i)/\Qcyc\right)^{\vee} \longrightarrow H^2_{\Iw}\left(G_S(\Qcyc), T_f(i)\right) \longrightarrow \bigoplus_{v \in S(\Qcyc)} H^0\left(\QQ_{\cyc,v}, A_\of(k-i)\right)^\vee.
\end{equation}
Therefore, the lemma follows from \textbf{(H0)}.
\end{proof}

\begin{theorem}\label{thm:Wing2}
Suppose that $f$ is non-ordinary at $p$ and {\textbf{\upshape{(H0)}} is valid}.
We have a pseudo-isomorphism of $\Lambda$-modules
\[
\left(\Sel(A_f(i)/\Qcyc)^\vee\right)_\tor\sim \left(\Sel_0(A_\of(k-i)/\Qcyc)^\vee\right)^\iota.
\]
\end{theorem}

\begin{proof}
Consider the defining sequence of the Selmer group
\[ 0 \longrightarrow \Sel(A_f(i)/\Qcyc) \longrightarrow H^1(\Qcyc, A_f(i)) \longrightarrow \bigoplus_{w\in S(\Qcyc)} \frac{H^1(\mathbb{Q}_{\mathrm{cyc},w}, A_f(i))}{H^1_\mathrm{f}(\mathbb{Q}_{\mathrm{cyc},w}, A_f(i))}. \]
By \cite[Proposition~3.8]{blochkato}, there is an isomorphism
\[\varprojlim H^1(L_{n,w_n},T_\of(k-i))\cong\left( \frac{H^1(\mathbb{Q}_{\mathrm{cyc},w}, A_f(i))}{H^1_\mathrm{f}(\mathbb{Q}_{\mathrm{cyc},w}, A_f(i))}
 \right)^\vee,\]
 where $w_n\in S(L_n)$ such that $w_n$ lies below $w$ and $w_{n+1}$. When $w\nmid p$, this is zero by \cite[\S17.10]{Kato}.
 When $w|p$, this  is also zero by \cite[Theorem 0.6]{perrinriou00b} under the hypothesis that $f$ is non-ordinary at $p$.  Hence, we have an isomorphism
 \[ \Sel(A_f(i)/\Qcyc) \cong H^1(\Qcyc, A_f(i))\]
 of $\Lambda$-modules.
 Combining this with Theorem \ref{thm:Wing}, we obtain a pseudo-isomorphism \[ \left(\Sel(A_f(i)/\Qcyc)^\vee\right)_\tor \sim \left(\tZ\right)^\iota\]
 of $\Lambda$-modules. But the latter is pseudo-isomorphic to $\Sel_0(A_f(i)/\Qcyc)^\vee$ by  Lemma \ref{lem:H2-Sel} under \textbf{\upshape{(H0)}}. This concludes the proof of the theorem.
\end{proof}

\section{Comparison of characteristic elements}
\subsection{Preliminary lemmas}
We prove several preliminary lemmas that will be used in the proofs of Theorems~\ref{lem:LS} and \ref{thm:LS} below.

\begin{lemma}\label{lem:quot-Z}
Assume that \textbf{\upshape{(H0)}} holds. If $F$ is an irreducible element of $\Lambda$ such that $F\nmid \ff_{\of,k-i}$, then we have a psuedo-isomorphism 
\[ (\cZ/\fZ)[F^\infty]\sim \Sel_0(A_f(i)/\Qcyc)^{\vee,\iota} [F^\infty] \]
of $\Lambda$-modules. In particular, 
$(\cZ/\fZ)[F]$ is finite if and only if $F^\iota\nmid \ff_{f,i}$.
\end{lemma}
\begin{proof}
We have the Poitou--Tate exact sequence
\[
0\rightarrow\fZ\rightarrow\cZ \rightarrow H^1(\Qcyc,A_\of(k-i))^\vee\rightarrow\Sel_0(A_\of(k-i)/\Qcyc)^\vee\rightarrow 0.
\]
 This gives the following short exact sequence
\[
0\rightarrow (\cZ/\fZ)\rightarrow H^1(\Qcyc,A_\of(k-i))^\vee\rightarrow\Sel_0(A_\of(k-i)/\Qcyc)^\vee \rightarrow 0 ,\] 
 which in turn induces the following exact sequence: 
\begin{equation}\label{eq:shortPT}
0\rightarrow (\cZ/\fZ)[F^\infty]\rightarrow H^1(\Qcyc,A_\of(k-i))^\vee[F^\infty]\rightarrow\Sel_0(A_\of(k-i)/\Qcyc)^\vee[F^\infty].    
\end{equation}
By assumption, $\Sel_0(A_\of(k-i)/\Qcyc)^\vee[F^\infty]$ is finite.
Therefore, we obtain a psuedo-isomorphism
\begin{equation}
(\cZ/\fZ)[F^\infty]\sim H^1(\Qcyc,A_\of(k-i))^\vee[F^\infty].
\end{equation}
By Theorem~\ref{thm:Wing} and Lemma~\ref{lem:H2-Sel}, we have the pseudo-isomorphism of $\Lambda$-modules
\[
H^1(\Qcyc,A_\of(k-i))^\vee[F^\infty]\sim  \Sel_0(A_f(i)/\Qcyc)^{\vee,\iota} [F^\infty].
\]
Combining these two pseudo-isomorphisms finishes the proof of the lemma.
\end{proof}

\begin{lemma}\label{lem:rank1-finite}
Assume that \textbf{\upshape{(H-IMC)}} holds. Let $F$ be an irreducible element of $\Lambda$ such that $F\nmid \eta_{k,i}$. Recall that $\loc_p$ denotes the localization map from $\fZ$ to $\cZ$. Then one has a pseudo-isomorphism 
\[ \left(\Lambda^2/\Lambda(\ff_{\of,k-i}^\sharp,\ff_{\of,k-i}^\flat)\right)[F^\infty] \sim  \left(\cZ/(\loc_p(\bZ_{f,i}))\right) [F^\infty] \]
of $\Lambda$-modules. Hence, it follows that $F\nmid \gcd(\ff_{\of,k-i}^\sharp,\ff_{\of,k-i}^\flat)$ if and only if $ \left(\cZ/(\loc_p(\bZ_{f,i}))\right)[F]$ is finite.
\end{lemma}
\begin{proof}
By \eqref{eq:Col-pair} and \eqref{eq:IMC}, we have the following exact sequence
\[
0\rightarrow \cZ/(\loc_p(\bZ_{f,i}))\rightarrow \Lambda^2/\Lambda(\ff_{\of,k-i}^\sharp,\ff_{\of,k-i}^\flat)\rightarrow \Lambda/\eta_{k,i}\rightarrow 0.    
\]
This gives the exact sequence
\[
0\rightarrow \left(\cZ/(\loc_p(\bZ_{f,i}))\right)[F^\infty]\rightarrow \left(\Lambda^2/\Lambda(\ff_{\of,k-i}^\sharp,\ff_{\of,k-i}^\flat)\right)[F^\infty]\rightarrow \Lambda/\eta_{k,i}[F^\infty].    
\]
The last term is finite since $F\nmid \eta_{k,i}$. Hence, the result follows.
\end{proof}
\subsection{Comparison between characteristic ideals of fine Selmer groups and signed Selmer groups}

The goal of this section is to prove a generalization of \cite[Theorem~1.2]{leisuj} (see Theorem~\ref{thm:LS} below). We shall do so via the following intermediate result:

\begin{theorem} \label{lem:LS}
Assume that $\textbf{\upshape{(H-IMC)}}$ and \textbf{\upshape{(H0)}} hold.  Let $F$ be an irreducible element of $\Lambda$ such that  $F\nmid \eta_{k,i}$ and that $F\nmid \ff_{\of,k-i}$. There is a pseudo-isomorphism of $\Lambda$-modules
\[ \left(\Lambda^2/\Lambda(\ff_{\of,k-i}^\sharp,\ff_{\of,k-i}^\flat)\right)[F^\infty] \sim \Sel_0(A_f(i)/\Qcyc)^{\vee,\iota} [F^\infty]. \]
 In particular,  $F^\iota\nmid \ff_{f,i}$  if and only if  $F\nmid\gcd\left(\ff_{\of,k-i}^\sharp,\ff_{\of,k-i}^\flat\right)$.
\end{theorem}
\begin{proof}
By Lemma \ref{lem:rank1-finite}, we have a pseudo-isomorphism 
\[ \left(\Lambda^2/\Lambda(\ff_{\of,k-i}^\sharp,\ff_{\of,k-i}^\flat)\right)[F^\infty] \sim  \left(\cZ/(\loc_p(\bZ_{f,i}))\right) [F^\infty] \]
of $\Lambda$-modules. Consider the short exact sequence
\[
    0\rightarrow \fZ/\Lambda\bZ_{f,i}\rightarrow \cZ/(\loc_p(\bZ_{f,i}))\rightarrow \cZ/\fZ\rightarrow0.
\]
Since $\fZ/\Lambda\bZ_{f,i}$ is a torsion $\Lambda$-module with  $\Char_\Lambda\left(\fZ/\Lambda\bZ_{f,i}\right) = \ff_{\of,k-i}$ not divisible by $F$, we have the following pseudo-isomorphism of $\Lambda$-modules
\[\left(\cZ/(\loc_p(\bZ_{f,i}))\right)[F^\infty]\sim  \left(\cZ/\fZ\right)[F^\infty] .\]
 By Lemma \ref{lem:quot-Z}, the latter is pseudo-isomorphic to $\Sel_0(A_f(i)/\Qcyc)^{\vee,\iota} [F^\infty]$. On combining these pseudo-isomorphisms, the theorem follows.
\end{proof}

We can now state and prove the main result of the article.

\begin{theorem}\label{thm:LS}
Assume that $\textbf{\upshape{(H-IMC)}}$ and \textbf{\upshape{(H0)}} hold. Let $F$ be an irreducible element of $\Lambda$ such that  $F\nmid \eta_{k,i}$. Then  $F^\iota\nmid \ff_{f,i}$ and $F\nmid \ff_{\of,k-i}$ if and only if $F\nmid\gcd\left(\ff_{\of,k-i}^\sharp,\ff_{\of,k-i}^\flat\right)$. 
\end{theorem}
\begin{proof}
  Theorem \ref{lem:LS} tells us that if $F^\iota\nmid\ff_{f,i}$ and $F\nmid \ff_{\of,k-i}$, then  $F\nmid\gcd\left(\ff_{\of,k-i}^\sharp,\ff_{\of,k-i}^\flat\right)$. This proves the "only if" implication.

  Conversely, suppose that $F\nmid\gcd\left(\ff_{\of,k-i}^\sharp,\ff_{\of,k-i}^\flat\right)$. From the inclusion
\[
\Sel_0(A_\of(k-i))\subset \Sel^\bullet(A_\of(k-i)),
\] we have $F\nmid \ff_{\of,k-i}$. The remaining assertion that $F^\iota\nmid \ff_{f,i}$ now follows from Theorem \ref{lem:LS}.
\end{proof}
\begin{remark}
Let $E/\QQ$ be an elliptic curve with good supersingular reduction at $p$. In \cite{leisuj}, it is stated in the proof of Proposition~3.1 that  by \cite[Corollary~2.5]{wingberg}, we have the equality
\[
 \Char_\Lambda\left(\Sel_{p^\infty}(E/\Qcyc)^\vee\right)_\tor=\Char_\Lambda\Sel_{0}(E/\Qcyc)^\vee.
\]
However, one of the two Selmer groups should be twisted by $\iota$ in order for the equality to hold. As such, for the rest of the proof to go through,  the additional hypothesis that the irreducible element $f$ (not to be confused with the notation for a modular form in the present article) satisfies $(f)=(f^\iota)$ is required. Consequently, the statement of Theorem~1.2 in op. cit. should also be modified. 
\end{remark}
Specializing our Theorem~\ref{thm:LS} to the case $F=\varpi$, where $\varpi$ is a uniformizer of $\cO$, we may deduce the following:

\begin{corollary}
\label{cor:LS}
Assume that $\textbf{\upshape{(H-IMC)}}$ and \textbf{\upshape{(H0)}} hold. Then the following statements are equivalent.

(a) The $\mu$-invariants of $\Sel_0(A_f(i)/\Qcyc)^\vee$ and $\Sel_0(A_\of(k-i)/\Qcyc)^\vee$ are zero.

(b) At least one of $\Sel^\sharp(A_\of(k-i)/\Qcyc)^\vee$ and $\Sel^\flat(A_\of(k-i)/\Qcyc)^\vee$ has trivial $\mu$-invariant.
\end{corollary}
\begin{remark}
Recall from \cite[Conjecture~A]{CoatesSujatha_fineSelmer}, \cite[Conjecture~A]{jhasuj} and \cite[Conjecture~1.2]{aribam2014mu} that the $\mu$-invariants of $\Sel_0(A_f(i)/\Qcyc)^\vee$ and $\Sel_0(A_\of(k-i)/\Qcyc)^\vee$ are predicted to always vanish. Corollary~\ref{cor:LS} gives an alternative formulation of these conjectures in terms of signed Selmer groups. In fact, numerical calculations carried out by Pollack (see \cite[\S7]{pollack03} and \url{http://math.bu.edu/people/rpollack/Data/data.html}) suggest that the $\mu$-invariants of both $\Sel^\sharp(A_\of(k-i)/\Qcyc)^\vee$ and $\Sel^\flat(A_\of(k-i)/\Qcyc)^\vee$ are zero when $A_\of(k-i)$ is given by $E[p^\infty]$ for some elliptic curve $E/\QQ$ with good supersingular reduction at $p$.
\end{remark}

\section*{acknowledgement}
We would like to thank the anonymous referee
for providing various helpful comments that have helped us improve the exposition of this article.


\bibliographystyle{abbrv}
\bibliography{references}

\end{document}